\documentclass[11pt]{article}
\usepackage{amsmath, amsthm, amssymb, amsfonts, mathrsfs}
\usepackage[dvips]{graphicx}
\newtheorem{theorem}{Theorem}[section]
\newtheorem{cor}[theorem]{Corollary}

\newtheorem{lem}[theorem]{Lemma}
\newtheorem{prop}[theorem]{Proposition}

\def \Zl {{\mathbb Z}}
\def \Nl {{\mathbb N}}
\def \Rl {{\mathbb R}}
\def \Ql {{\mathbb Q}}
\def \Cl {{\mathbb C}}

\def \vl {{\mathbf v}}

\title{ Pretty Good State Transfer on Circulant Graphs}

\author{ Hiranmoy Pal\\
Department of Mathematics\\
Indian Institute of Technology Guwahati\\Guwahati, India - 781039\\
Email: hiranmoy@iitg.ernet.in\\
\\
Bikash Bhattacharjya\\
Department of Mathematics\\
Indian Institute of Technology Guwahati\\Guwahati, India - 781039\\
Email: b.bikash@iitg.ernet.in
}
\begin{document}
\maketitle

\vspace{-0.3in}

\begin{center}{Abstract}\end{center}
Let $G$ be a graph with adjacency matrix $A$. The transition matrix of $G$ relative to $A$ is defined by $H(t):=\exp{\left(-itA\right)},\;t\in\Rl$. The graph $G$ is said to admit pretty good state transfer between a pair of vertices $u$ and $v$ if there exists a sequence of real numbers $\{t_k\}$ and a complex number $\gamma$ of unit modulus such that $\lim\limits_{k\rightarrow\infty} H(t_k) e_u=\gamma e_v.$ We find that pretty good state transfer occurs in a cycle on $n$ vertices if and only if $n$ is a power of two and it occurs between every pair of antipodal vertices. In addition, we look for pretty good state transfer in more general circulant graphs. We prove that union (edge disjoint) of an integral circulant graph with a cycle, each on $2^k$ $(k\geq 3)$ vertices, admits pretty good state transfer. The complement of such union also admits pretty good state transfer. This enables us to find some non-circulant graphs admitting pretty good state transfer. Among the complement of cycles we also find a class of graphs not exhibiting pretty good state transfer. 

\section{Introduction}
Perfect state transfer (PST) has great significance due to its applications in quantum information processing and cryptography (see \cite{ben,cha,ek}). The phenomenon of PST in quantum communication networks was originally introduced by Bose in \cite{bose}. In mathematical terms, PST is defined as follows. The transition matrix of a graph $G$ with adjacency matrix $A$ is defined by
\[H(t):=\exp{\left(-itA\right)},\;t\in\Rl.\]
The transition matrix $H(t)$ is a unitary matrix and it is also a polynomial in $A$. The graph $G$ is said to exhibit PST from a vertex $u$ to another vertex $v$ if there exists a real number $\tau$ and a complex number $\gamma$ with $|\gamma|=1$ such that $H(\tau)e_u=\gamma e_v.$ In case $H(\tau)e_u=\gamma e_u$, we say that $G$ is periodic at the vertex $u$. Moreover, $G$ is said to be periodic if there is $\tau\in\Rl$ such that $H(\tau)$ is a scalar multiple of the identity matrix, in which case the graph is periodic at all vertices. Finding whether a given graph has PST is quite difficult especially when the graph is large. Remarkably, in \cite{cou1}, Coutinho \emph{et al.} showed that one can decide whether a graph admits PST in polynomial time with respect to the size of the graph on a classical computer.\par
In initial papers (see \cite{chr1,chr2}), we find that PST occurs on Cartesian powers of a path on two vertices and a path on three vertices. Further the results have been generalized to NEPS of the respective graphs (see\cite{ber,chi,pal}). In \cite{god1}, we find that a regular graph is periodic if and only if its eigenvalues are integers. As a result, if PST occurs on a Cayley graph then it must be integral. A characterization of PST in integral circulent graphs appears in \cite{mil4}. Some results on PST in gcd-graph, which is a special class of integral Cayley graph, can be found in \cite{pal1,pal2}. In \cite{god2}, we see that there are only finitely many connected graphs with maximum valency at most $k$ where PST occurs. Since there are less number of graphs having PST, we consider a relaxation to PST called pretty good state transfer (PGST).\par
The notion of PGST was first introduced by Godsil in \cite{god1}. A graph $G$ with transition matrix $H(t)$ has PGST between a pair of vertices $u$ and $v$ if there is a sequence of real numbers $\{t_k\}$ and a complex number of unit modulus $\gamma$ such that \[\lim_{k\rightarrow\infty} H(t_k) e_u=\gamma e_v.\]
In such a case, we also say that $G$ exhibits PGST from $u$ to $v$ with respect to the sequence $\{t_k\}$. This is equivalent to say that for $\epsilon>0$, there exists $t\in \Rl$ and $\gamma\in\Cl$ with $|\gamma|=1$ such that
\[\left| e_u^T H(t) e_v -\gamma\right|<\epsilon.\]
There are a few published papers which discuss PGST. Godsil \emph{et al.} \cite{god4} showed that the path $P_n$ exhibits PGST if and only if $n+1$ equals to either $2^m$ or $p$ or $2p$, where $p$ is an odd prime. We also see in \cite{fan} that a double star $S_{k,k}$ admits PGST if and only if $4k+1$ is not a perfect square. The double star can be realized as a corona product of the complete graph $K_2$ and an empty graph. PGST in more general corona products has been studied in \cite{ack,ack1}. Moreover, in \cite{pal3}, we find some NEPS of the path on three vertices having PGST. Some other relevant results regarding PST and PGST can be found in \cite{cou,cou2,god3,kirk}.\par
Circulant graphs arises frequently in communication networks. Among the circulant graphs only integral circulant graphs are periodic (see \cite{god1}). It can be proved that if a graph is periodic then the graph has PGST if and only if it has PST. Since a complete characterization of PST in integral circulant graph is known, we consider PGST in circulant graphs which are not integral. In the present article, we completely classify which cycles exhibit PGST. Beside cycles, we also find two classes of non-integral circulant graphs one of which exhibits PGST and the other does not. Apart from circulant graphs, we use Cartesian product to find some non-circulant graphs having PGST.\par
Let $\left(\Gamma,+\right)$ be a finite abelian group and consider $S\setminus\{0\}\subseteq\Gamma$ with $\left\lbrace -s:s\in S\right\rbrace=S$. Such a set $S$ is called a symmetric subset of $\Gamma$. A Cayley graph over $\Gamma$ with a symmetric set $S$ is denoted by $Cay\left(\Gamma,S\right)$. The graph has the vertex set $\Gamma$ where two vertices $a,b\in\Gamma$ are adjacent if and only if $a-b\in S$. The set $S$ is called the connection set of $Cay\left(\Gamma,S\right)$. Let $\Zl_n$ be the cyclic group of order $n$. A circulant graph is a Cayley graph over $\Zl_n$. A cycle $C_n$, in particular, is a circulant graph over $\Zl_n$ with the connection set $\{1,n-1\}$. Eigenvalues and eigenvectors of a cycle are very well known. Suppose $\omega_n=\exp{\left(\frac{2\pi i}{n}\right)}$ is the primitive $n$-th root of unity. Then the eigenvalues of $C_n$ are
\begin{eqnarray}\label{e}
\lambda_l=2\cos{\left(\frac{2l\pi}{n}\right)},\;0\leq l\leq n-1,
\end{eqnarray}
and the corresponding eigenvectors are $\vl_l=\left[1,\omega_n^l,\ldots,\omega_n^{l(n-1)}\right]^T$.\par
The eigenvalues and eigenvectors of a Cayley graph over an abelian group are also known in terms of characters of the abelian group. In \cite{wal1}, it appears that the eigenvectors of a Cayley graph over an abelian group are independent of the connection set. Consider two symmetric subsets $S_1,S_2$ of $\Gamma$. So the set of eigenvectors of both $Cay(\Gamma, S_1 )$ and $Cay(\Gamma, S_2 )$ can be chosen to be equal. Hence we have the following result.
\begin{prop}\label{3ab}
If $S_1$ and $S_2$ are symmetric subsets of an abelian group $\Gamma$ then adjacency matrices of the Cayley garphs $Cay(\Gamma, S_1 )$ and $Cay(\Gamma, S_2 )$ commute.
\end{prop}
If $d$ is a proper divisor of $n$ then we define
\[S_n(d)=\{x\in\Zl_n: gcd(x,n)=d\},\]
and for any set $D$ containing proper divisors of $n$, we define
\[S_n(D)=\bigcup\limits_{d\in D} S_n(d).\]
The set $S_n(D)$ is called a gcd-set of $\Zl_n$. A gcd graph over $\Zl_n$ is a circulant graph whose connection set is a gcd set. We denote a gcd graph with the connection set $S_n(D)$ by $G(n,D)$.\par
A graph is called integral if all its eigenvalues are integers. The following theorem determines circulant graphs which are integral.
\begin{theorem}\cite{so}
A circulant graph is integral if and only if the connection set is a gcd-set.
\end{theorem}
The Cartesian product of two graphs $G_{1}$ and $G_{2}$ with vertex sets $V_{1}$ and $V_{2}$ is the graph $G_{1}\square G_{2}$, with vertex set $V_{1}\times V_{2}$. Two vertices $(u_{1},u_{2})$ and $(v_{1},v_{2})$ are adjacent in $G_{1}\square G_{2}$ if and only if either $u_{1}$ is adjacent to $v_{1}$ in $G_{1}$ and $u_{2}=v_{2}$, or $u_{1}=v_{1}$ and $u_{2}$ is adjacent to $v_{2}$ in $G_{2}$. The transition matrix of a Cartesian product of two graphs is given by the following result.
\begin{lem}\cite{chr2}
Let $G_1$ and $G_2$ be two graphs having transition matrices $H_{G_1}(t)$ and $H_{G_2}(t)$, respectively. Then the transition matrix of $G_1\square G_2$ is $H_{G_1\square G_2}(t)=H_{G_1}(t)\otimes H_{G_2}(t)$.
\end{lem}
Now we introduce Kronecker approximation theorem on simultaneous approximation of numbers. This will be used later to find graphs allowing PGST. 
\begin{theorem}[Kronecker approximation theorem]\cite{apo}
If $\alpha_1,\ldots,\alpha_l$ are arbitrary real numbers and if $1,\theta_1,\ldots, \theta_l$ are real, algebraic numbers linearly independent over $\Ql$ then for $\epsilon>0$ there exist $q\in\Zl$ and $p_1,\ldots,p_l\in\Zl$ such that
\[\left|q\theta_j-p_j-\alpha_j\right|<\epsilon.\]
\end{theorem}
A bound on $q$ is given in \cite{mala} relative to the precision $\epsilon$ and some other given constraints.

\section{Pretty Good State Transfer on Circulant Graphs}

We begin with the discussion that the odd cycles never exhibit PGST. Moreover, if an even cycle admits PGST then it must occur only between the antipodal vertices. Suppose $G$ is a graph with adjacency matrix $A$. If $P$ is the matrix of an automorphism of $G$ then $P$ must commute with $A$ and consequently $P$ commutes with the transition matrix $H(t)$. Suppose $G$ allows PGST between two vertices $u$ and $v$. Then there exists a sequence of real numbers $\{t_k\}$ and a complex number $\gamma$ of unit modulus such that
\[\lim_{k\rightarrow\infty} H(t_k) e_u=\gamma e_v.\]
Further this implies that
\[\lim_{k\rightarrow\infty} H(t_k) \left(Pe_u\right)=\gamma \left(Pe_v\right).\]
Since the sequence $\{H(t_k)e_u\}$ cannot have two different limits, we conclude that if $P$ fixes $e_u$ then $P$ must fix $e_v$ as well. As a consequence, we have the following result.
\begin{lem}\label{l0}
If pretty good state transfer occurs in a cycle $C_{n}$ or in the complement of $C_{n}$ then $n$ is even and it occurs only between the pair of vertices $u$ and $u+\frac{n}{2}$, where $u,u+\frac{n}{2}\in\Zl_n$.
\end{lem}
From now onwards we only consider the even cycles. Notice that it is enough to find PGST in $C_{n}$ between the pair of vertices $0$ and $\frac{n}{2}$. We thus calculate the $(0,\frac{n}{2})$-th entry of the transition matrix of $C_n$. If $\sum\limits_{l=0}^{n-1}\lambda_lE_l$ is the spectral decomposition of the adjacency matrix $A$ of $C_n$ then the transition matrix can be calculated as
\[H(t)=\exp{\left(-itA\right)}=\sum\limits_{l=0}^{n-1}\exp{\left(-i\lambda_l t\right)E_l}.\]
Using $\left(\ref{e}\right)$, we evaluate the $(0,\frac{n}{2})$-th entry of $H(t)$ as 
\begin{eqnarray}\label{e0}
H(t)_{0,\frac{n}{2}}=\frac{1}{n}\sum\limits_{l=0}^{n-1}\exp{\left(-i\lambda_l t\right)}\cdot\omega_{n}^{\frac{nl}{2}}=\frac{1}{n}\sum\limits_{l=0}^{n-1}\exp{\left[-i\left(\lambda_l t+l\pi\right)\right]}.
\end{eqnarray}
In the following result, we distinguish a class of cycles admitting PGST. Later, we will show that these are the only cycles with PGST.
\begin{lem}\label{l1}
A cycle $C_{n}$ exhibits pretty good state transfer if $n=2^k,\;k\geq 2$, with respect to a sequence in $2\pi\Zl.$ 
\end{lem}
\begin{proof}
It is well known that a cycle on four vertices has PST between the antipodal vertices. Since PGST is a generalization of PST, it is enough to consider the case $n=2^k$ where $k\geq 3.$ First we show that the distinct positive eigenvalues of $C_n$ are linearly independent over $\Ql$. The eigenvalues of $C_n$ can be realized as
\[\lambda_l=2\cos{\left(\frac{2l\pi}{n}\right)}=\omega_n^l+\omega_n^{-l}.\]
It is well known that the minimal polynomial of $\omega_n$ over $\Ql$ has degree $\phi(n)$, where $\phi$ is Euler's phi-function (see \cite{esc}). It is evident that $\lambda_l$ is positive only when $-\frac{\pi}{2}<\frac{2l\pi}{n}<\frac{\pi}{2}$. Consequently, the distinct positive eigenvalues are $\lambda_l$ where $0\leq l<\frac{n}{4}.$ If the distinct positive eigenvalues are dependent over $\Ql$ then $\omega_n$ will be a root of a polynomial of degree at most $m=2^{k-1}-1$ as $\omega_n^{-1}=-\omega_n^{m}$. But since $2^{k-1}-1<2^{k-1}= \phi(2^k),$ we conclude that the distinct positive eigenvalues are linearly independent over $\Ql$. Thus the distinct positive eigenvalues are
\[\lambda_0,\lambda_1,\ldots,\lambda_{2^{k-2}-1}.\]
For $ 1\leq l\leq 2^{k-2}-1$, consider the following real numbers
\begin{eqnarray*}
\alpha_l=\begin{cases} 0, & \text{ if $l$ is even}\\
\frac{1}{2}, & \text{ if $l$ is odd.}
\end{cases}
\end{eqnarray*}
By Kronecker approximation theorem we find that for $\delta>0$ there exist $q,m_1,\ldots,m_{2^{k-2}-1}\in\Zl$ such that for $l=1,\ldots, 2^{k-2}-1$
\begin{eqnarray}\label{e1}
\left|q\lambda_l-m_l-\alpha_l\right|<\frac{\delta}{2\pi},\; \emph{i.e,}\;\left|2q\pi\lambda_l-2m_l\pi-2\alpha_l\pi\right|<\delta.
\end{eqnarray} 
The graph $C_{2^k}$ is bipartite and each of its eigenvalue is repeated twice except $2$ and $-2$. For $ 1\leq l\leq 2^{k-2}-1,$ we see that
\[\lambda_l=2\cos{\left(\frac{2l\pi}{n}\right)}=2\cos{\left(2\pi-\frac{2l\pi}{n}\right)}=\lambda_{n-l}.\] Similarly, for $ 1\leq l\leq 2^{k-2}-1$ we have
\[\lambda_l=-\lambda_{\frac{n}{2}-l}=-\lambda_{\frac{n}{2}+l}.\]
Since $\lambda_0,\lambda_{2^{k-2}},\lambda_{2^{k-1}}$ and $\lambda_{3\cdot2^{k-2}}$ are integers, considering $t=2q\pi$, we observe that for each $l=0,\;2^{k-2},\;2^{k-1},\;3\cdot2^{k-2}$ there exists an integer $l'$ such that
\[\left|\left(\lambda_l t+l\pi\right)- 2l'\pi\right|<\delta.\]
Also for each $l=1,\ldots, 2^{k-2}-1$, considering $t=2q\pi$ and using $\left(\ref{e1}\right)$, we find
\begin{eqnarray*}
l'=\begin{cases} \frac{2m_l+l+1}{2}, & \text{ if $l$ is odd}\\
\frac{2m_l+l}{2}, & \text{ if $l$ is even.}
\end{cases}
\end{eqnarray*}
such that
\[\left|\left(\lambda_l t+l\pi\right)- 2l'\pi\right|<\delta.\]
Since the relation $\lambda_l=-\lambda_{\frac{n}{2}-l}=-\lambda_{\frac{n}{2}+l}=\lambda_{n-l}$ on eigenvalues of $C_{2^k}$ holds for $ 1\leq l\leq 2^{k-2}-1$, considering $t=2q\pi$, we conclude that for each $l=0,\ldots, n-1,$ there is an integer $l'$ such that
\[\left|\left(\lambda_l t+l\pi\right)- 2l'\pi\right|<\delta.\]
Therefore by uniform continuity of the exponential function $\exp{\left(-ix\right)}$, it follows that for $\epsilon>0$ there exists $q\in\Zl$ so that $t=2q\pi$ and $\left|\exp{\left[-i\left(\lambda_l t+l\pi\right)\right]}- 1\right|<\epsilon.$ Finally, from equation (\ref{e0}), we observe that
\[\left|H(t)_{0,\frac{n}{2}}-1\right|=\frac{1}{n}\left|\sum\limits_{l=0}^{n-1}\left(\exp{\left[-i\left(\lambda_l t+l\pi\right)\right]}-1\right)\right|<\epsilon.\]
This leads to the conclusion that $C_n$ admits PGST whenever $n$ is a power of two, with respect to a sequence in $2\pi\Zl$.
\end{proof}
Now using Lemma \ref{l1} we find more general circulant graphs allowing PGST. We present this as a theorem. 
\begin{theorem}\label{t1}
Let $n=2^k$ with $k\geq 3$. If $D$ is a set of proper divisors of $n$ not containing $1$ then the circulant graph $C_n\cup G(n,D)$ as well as its complement admit pretty good state transfer with respect to the same sequence in $2\pi\Zl$. 
\end{theorem} 
\begin{proof}
Notice that both graphs $C_n$ and $G(n,D)$ have the same vertex set $\Zl_n$ but the edge sets are disjoint as $1\not\in D$. Suppose $A$ and $B$ are the adjacency matrices of $C_n$ and $G(n,D)$, respectively. The adjacency matrix of $C_n\cup G(n,D)$ is therefore $A+B$. Since the matrices $A$ and $B$ commute, the transition matrix of $C_n\cup G(n,D)$ is
\[H(t)=H_A(t) H_B(t),\]  
where $H_A(t)$ and $H_B(t)$ are transition matrices of $C_n$ and $G(n,D)$, respectively. Suppose $B$ has the spectral decomposition $\sum\limits_{j=1}^r \theta_j E_j$. Since the graph $G(n,D)$ is integral, the values $\theta_j$'s are integers. This in turn implies that if $t\in 2\pi\Zl$ then 
\[H_B(t)=\exp{\left(-itB\right)}=\sum\limits_{j=1}^r \exp{(-it\theta_j)}E_j=I.\]
Hence $H(t)=H_A(t)$ whenever $t\in 2\pi\Zl$. Finally, by Lemma \ref{l1}, we observe that $C_n\cup G(n,D)$ exhibits PGST.\par
It remains to show that complement of $C_n\cup G(n,D)$ also admits PGST. The adjacency matrix of the complement is $J-I-(A+B).$ Since all circulant graphs are regular, the matrix $A+B$ commutes with $J-I$. If $H'(t)$ is the transition matrix of the complement then
\[H'(t)=\exp{\left(-it\left(J-I\right)\right)}H(-t).\]
Clearly the eigenvalues of $J-I$ are integers. Hence following the same argument as given in the previous part, we have the desired result.
\end{proof}
In Theorem \ref{t1}, the graph $C_n\cup G(n,D)$ can never be a gcd graph and so it cannot be integral. The reason is that if $C_n\cup G(n,D)=G(n,D')$ for some divisor set $D'$ then $1\in D'$. Therefore $\{1,n-1\}\cup S_n(D)$ contains all odd numbers in $\Zl_n$ as $\{1,n-1\}\cup S_n\left(D\right)=S_n\left(D'\right).$ As $n=2^k\;(k\geq 3)$, we see, for example, that $\{1,n-1\}\cup S_n(D)$ can never contain $3$.\par
Notice that if $D$ is empty then Theorem \ref{t1} implies that complement of a cycle also allows PGST. Another thing to observe in the proof of Theorem \ref{t1} is that for a fixed $n$, there exist a fixed sequence with respect to which all graphs of the form $C_n\cup G(n,D)$ as well as their complement exhibit PGST. \par
It turns out that there are some more graphs allowing PGST apart from the circulant graphs we already mentioned. Before finding those graphs we introduce the following notations. For $k\geq 3$, we denote
\[\mathcal{G}_k = \left\lbrace C_{2^k}\cup G\left(2^k,D\right)\;:\; D \text{ is a set of proper divisors of } 2^k \text{ and } 1\not\in D  \right\rbrace.\]
Also the set of the complements of graphs in $\mathcal{G}_k$ is denoted by $\bar{\mathcal{G}_k}$. Further we set
\[\mathcal{G}=\bigcup\limits_{k\geq 3}\left(\mathcal{G}_k\cup\bar{\mathcal{G}_k}\right).\]
Now we find the following corollaries regarding PGST in Cartesian products.
\begin{cor}\label{c0}
Let $G_1, G_2\in\mathcal{G}_k\cup\bar{\mathcal{G}_k}$. Then the Cartesian product $G_1\square G_2$ as well as its complement admit pretty good state transfer. 
\end{cor}
\begin{proof}
In the proof of Theorem \ref{t1} we see that if $G_1, G_2\in\mathcal{G}_k\cup\bar{\mathcal{G}_k}$ then both $G_1$ and $G_2$ have PGST with respect to the same sequence $\{t_m\}$ in $2\pi\Zl$. Suppose $G_1$ admits PGST between the vertices $u_1$ and $v_1$ and $G_2$ admits PGST between the vertices $u_2$ and $v_2$. Also assume that $H_{G_1}(t)$ and $H_{G_2}(t)$ are the transition matrices of $G_1$ and $G_2$, respectively. Therefore there exist $\gamma_1,\gamma_2\in\Cl$ with $|\gamma_1|=|\gamma_2|=1$ such that
\[\lim_{m\rightarrow\infty} e_{u_1}^T H_{G_1}(t_m) e_{v_1}=\gamma_1 \text{ and } \lim_{m\rightarrow\infty} e_{u_2}^T H_{G_2}(t_k) e_{v_2}=\gamma_2.\]
Using the property of transition matrix of a Cartesian product, we have
\begin{eqnarray*}
\left(e_{u_1}\otimes e_{u_2}\right)^T\left(H_{G_1\square G_2}\left(t_m\right)\right)\left(e_{v_1}\otimes e_{v_2}\right)
& = & \left(e_{u_1}\otimes e_{u_2}\right)^T\left(H_{G_1}\left(t_m\right)\otimes H_{G_2}\left(t_m\right)\right)\left(e_{v_1}\otimes e_{v_2}\right)\\
& = & \left(e_{u_1}^T H_{G_1}\left(t_m\right)e_{v_1}\right)\cdot\left(e_{u_2}^T H_{G_2}\left(t_m\right)e_{v_2}\right).
\end{eqnarray*}
Now taking limits on both sides we find that $G_1\square G_2$ admits PGST.\par
It remains to show that the complement of $G_1\square G_2$ exhibits PGST. Notice that both $G_1$ and $G_2$ are regular graphs and therefore $G_1\square G_2$ is also a regular graph. Now as in the proof of second part of Theorem \ref{t1}, we have the desired conclusion.
\end{proof}
More generally, following the proof of Corollary \ref{c0}, we can deduce that if two graphs have PGST with respect to the same sequence then their Cartesian product also admits PGST with respect to that sequence. In the next result, we find another class of graphs exhibiting PGST.
\begin{cor}
Let a graph $G_1$ be periodic at a vertex at time $2\pi$. If $G_2\in\mathcal{G}$ then the Cartesian product $G_1\square G_2$ admits pretty good state transfer. If $G_1$ is regular then the complement of $G_1\square G_2$ also exhibits pretty good state transfer. 
\end{cor}
\begin{proof}
Suppose $G_1$ is periodic at a vertex $u$ at time $2\pi$. If $H_{G_1}(t)$ is the transition matrix of $G_1$ then there exists $\gamma_1\in\Cl$ with $|\gamma|=1$ such that $e_u^T H_{G_1}(2\pi)e_u=\gamma_1$. Hence for $q\in\Zl$, we have $e_u^T H_{G_1}(2q\pi)e_u=\gamma_1^q$. Since $G_2\in\mathcal{G}$ there is a sequence $\{t_m\}$ in $2\pi\Zl$ with respect to which $G_2$ exhibits PGST between two vertices $v$ and $w$, say. Since the unit circle is compact there is a subsequence $\{t'_m\}$ of $\{t_m\}$ such that $\left\lbrace e_u^T H_{G_1}(t'_m)e_u\right\rbrace$ is convergent. If $H_{G_2}(t)$ is the transition matrix of $G_2$ then
\begin{eqnarray*}
\left(e_{u}\otimes e_{v}\right)^T\left(H_{G_1\square G_2}\left(t'_m\right)\right)\left(e_{u}\otimes e_{w}\right)
= \left(e_{u}^T H_{G_1}\left(t'_m\right)e_{u}\right)\cdot\left(e_{v}^T H_{G_2}\left(t'_m\right)e_{w}\right).
\end{eqnarray*}
Now taking limits on both sides, we find that $G_1\square G_2$ admits PGST.\par
In case $G_1$ is regular then $G_1\square G_2$ is also regular. Therefore the complement of $G_1\square G_2$ also exhibits PGST.
\end{proof}
\textbf{Remark:} If a graph is integral then it is periodic at $2\pi$. So Cartesian product of an integral graph and a graph in $\mathcal{G}$ allows PGST. This gives a large number of graphs having PGST.\par
In Lemma \ref{l1}, we found a class of cycles with PGST. Next we investigate PGST in the remaining class of cycles. The only possibility we need to consider is the case when $n$ has an odd prime factor. We have used some of the techniques from \cite{god4} to prove the following result. 
\begin{lem}\label{l2}
Let $m\in\Nl$ and $p$ be an odd prime such that $n=mp$. Then the cycle $C_{n}$ does not exhibit pretty good state transfer.
\end{lem}
\begin{proof}
Notice that if $m$ is an odd number then, by Lemma \ref{l0}, we have the desired result. Hereafter we assume that $m$ is even. For an odd prime $p$ we have the following identity involving the primitive $p$-th root $\omega_p$ of unity:
\[1+\omega_p+\omega_p^2+\ldots+\omega_p^{p-1}=0.\]
This further yields
\begin{eqnarray}\label{e2}
1+2\sum\limits_{r=1}^{\frac{p-1}{2}}\cos{\left(\frac{2r\pi}{p}\right)}=0.
\end{eqnarray}
Multiplying both sides of $\left(\ref{e2}\right)$ by $2\cos{\left(\frac{2\pi}{n}\right)}$ we obtain the following relation of eigenvalues of $C_n$ (as given in $\left(\ref{e}\right)$).
\begin{eqnarray}\label{e3}
\lambda_1 + \sum\limits_{r=1}^{\frac{p-1}{2}}\lambda_{mr+1} + \sum\limits_{r=1}^{\frac{p-1}{2}} \lambda_{mr-1}=0.
\end{eqnarray}
Similarly multiplying $\left(\ref{e2}\right)$ by $2\cos{\left(\frac{4\pi}{n}\right)}$ gives
\begin{eqnarray}\label{e4}
\lambda_2 + \sum\limits_{r=1}^{\frac{p-1}{2}}\lambda_{mr+2} + \sum\limits_{r=1}^{\frac{p-1}{2}} \lambda_{mr-2}=0.
\end{eqnarray}
Now from equation $\left(\ref{e3}\right)$ and $\left(\ref{e4}\right)$ we get
\begin{eqnarray}\label{e5}
\left(\lambda_2 - \lambda_1\right) + \sum\limits_{r=1}^{\frac{p-1}{2}}\left(\lambda_{mr+2} - \lambda_{mr+1}\right) + \sum\limits_{r=1}^{\frac{p-1}{2}}\left(\lambda_{mr-2}-\lambda_{mr-1}\right)=0.
\end{eqnarray}
If $C_n$ admits PGST then, by equation $\left(\ref{e0}\right)$, we have a sequence of real numbers $\{t_k\}$ and a complex number $\gamma$ with $|\gamma|=1$ such that 
\[\lim_{k\rightarrow\infty}\sum\limits_{l=0}^{n-1}\exp{\left[-i\left(\lambda_l t_k+l\pi\right)\right]}=n\gamma.\]
Since the unit circle is compact, we have a subsequence $\{t'_k\}$ of $\{t_k\}$ such that
\[\lim_{k\rightarrow\infty}\exp{[-i\left(\lambda_l t'_k+l\pi\right)}=\gamma,\text{ for }0\leq l\leq n-1.\] 
This further implies that
\[\lim_{k\rightarrow\infty}\exp{[-i\left(\lambda_{l+1}-\lambda_l\right)t'_k]}=-1.\]
Denoting the term in left hand side of equation $\left(\ref{e5}\right)$ as $L$, it is evident that
\[\lim_{k\rightarrow\infty}\exp{\left(-iLt'_k\right)}=-1.\]
But this is not possible as $L=0$. Hence there is no pretty good state transfer in $C_n$ whenever $n$ has an odd prime factor.
\end{proof}
So far we have developed a complete characterization for PGST on cycles. We state the result as a theorem.
\begin{theorem}
A cycle $C_n$ admits pretty good state transfer if and only if $n=2^k$ for $k\geq 2.$
\end{theorem}
In the next result we see that complement of all cycles do not posses PGST. This gives an another class of circulant graphs not allowing PGST. We provide this result as a corollary.   
\begin{cor}\label{c2}
Let $m\in\Nl$ with $m\neq 2$ such that $n=mp$ for some odd prime $p$. Then the complement of the cycle $C_{n}$ does not exhibit pretty good state transfer.
\end{cor}
\begin{proof}
For $m=1$, by Lemma \ref{l0}, we conclude that complement of $C_{n}$ does not admit PGST. 
So we only consider the case $m\geq 3$.\par
The cycles are regular graphs. Therefore the eigenvalues of the complement of $C_{n}$ are $\lambda'_0=n-\lambda_0-1$ and $\lambda'_l=-\lambda_l-1,$ whenever $1\leq l\leq n-1,$ corresponding to the same set of eigenvectors as that of $C_{n}$. This means that $\left(0,\frac{n}{2}\right)$-th entry of the transition matrix of the complement graph can be obtained from equation $\left(\ref{e0}\right)$ by replacing the eigenvalues $\lambda_l$ with $\lambda'_l$. Now along the line of proof of Lemma \ref{l2}, we find that
\[\lambda'_1 + \sum\limits_{r=1}^{\frac{p-1}{2}}\lambda'_{mr+1} + \sum\limits_{r=1}^{\frac{p-1}{2}} \lambda'_{mr-1}=-p.\]
Since $m\geq 3$, we also have
\[\lambda'_2 + \sum\limits_{r=1}^{\frac{p-1}{2}}\lambda'_{mr+2} + \sum\limits_{r=1}^{\frac{p-1}{2}} \lambda'_{mr-2}=-p.
\]
The above two identities gives
\begin{eqnarray*}
\left(\lambda'_2 - \lambda'_1\right) + \sum\limits_{r=1}^{\frac{p-1}{2}}\left(\lambda'_{mr+2} - \lambda'_{mr+1}\right) + \sum\limits_{r=1}^{\frac{p-1}{2}}\left(\lambda'_{mr-2}-\lambda'_{mr-1}\right)=0,
\end{eqnarray*}
which is similar to equation $\left(\ref{e5}\right)$. Hence, following the same argument as given in Lemma \ref{l2}, we conclude that the complement of $C_n$ does not exhibit PGST. 
\end{proof}

\section{Conclusions}
In the past decade the study of PST in graphs has received considerable attention. Now we know a handful of graphs exhibiting PST. It is always preferable to find graphs having PST between vertices at a long distance. So far the best known graphs in this regard are the hypercubes. In a hypercube with $n$ vertices, we have PST between vertices at a distance $\log_2(n)$. It is thus desirable to have graphs allowing PST between vertices at a distance of $O(n).$ Most lucrative classes of graphs in this regard are the paths $P_n$ and the cycles $C_n$ as both of them have large diameter. But it is well known that $P_n$ does not exhibit PST whenever $n\geq 4$ and $C_n$ admits PST only when $n=4$.\par
Meanwhile the study of PGST got some interest. In \cite{god4}, the authors presented a remarkable result which classifies the paths $P_n$ admitting PGST between the end vertices. This serves as an example where PGST takes place between vertices at a distance $n$. In this article, we found that $C_{n}$ exhibits PGST if and only if $n$ is a power of two and PGST occurs between any pair of antipodal vertices. This gives an another class of graphs having PGST between vertices at a distance of $O(n).$ We also found a good number of circulant graphs allowing or not allowing PGST. Apart from the circulant graphs, we found some other graphs allowing PGST.\par
There are a few scopes for further research in this direction. We found that the circulant graph $C_{2^k}\cup G\left(2^k,D\right)$ for $k\geq3,1\not\in D,$ admits PGST. Thus one may try to find if there are any other circulant graphs allowing PGST. If we go through the proof of Lemma \ref{l2} and Corollary \ref{c2}, we see that the complement of $C_{2p}$, where $p$ is a prime, does not have PGST with respect to any sequence in $\pi\Zl.$ In that case, it would be interesting to find if the complement of those cycles have PGST at all. Moreover, it is desirable to classify which circulant graphs exhibit PGST. More generally, it is preferable to have a characterization of PGST in Cayley graphs.


\begin{thebibliography}{99}

\bibitem{ack} E. Ackelsberg, Z. Brehm, A. Chan, J. Mundinger and C. Tamon, \emph{Laplacian State Transfer in Coronas}, Linear Algebra and its Applications, \textbf{506}:154-167 (2016).

\bibitem{ack1} E. Ackelsberg, Z. Brehm, A. Chan, J. Mundinger and C. Tamon, \emph{Quantum State Transfer in Coronas}, arXiv preprint arXiv:1605.05260 (2016).

\bibitem{apo} T. M. Apostol, \emph{Modular Functions and Dirichlet Series in Number Theory}, 2nd ed. New York: Springer-Verlag (1997).

\bibitem{mil4} M. Ba\v si\'c, \emph{Characterization of circulant networks having perfect state transfer}, Quantum inf. process, \textbf{12}:345-364 (2011).

\bibitem{ben} C. H. Bennett and G. Brassard, \emph{Quantum Cryptography: Public Key Distribution and Coin Tossing}, Proc. IEEE Int. Conf. Computers Systems and Signal Processing, Bangalore, India. 175-179 (1984).

\bibitem{ber} A. Bernasconi, C. Godsil and S. Severini, \emph{Quantum networks on cubelike graphs}, Physical Review A, \textbf{78}:052320 (2008).

\bibitem{bose} S. Bose, \emph{Quantum communication through an unmodulated spin chain}, Physical Review Letters, \textbf{91}(20):207901 (2003).

\bibitem{cha} R. J. Chapman, M. Santandrea, Z. Huang, G. Corrielli, A. Crespi, M. H. Yung, R. Osellame and A. Peruzzo, \emph{Experimental perfect state transfer of an entangled photonic qubit}, Nature communications, 7 (2016).

\bibitem{chr1} M. Christandl, N. Datta, A. Ekert and A. J. Landahl, \emph{Perfect state transfer in quantum spin networks}, Physical Review Letters, \textbf{92}:187902 (2004).

\bibitem{chr2} M. Christandl, N. Datta, T. Dorlas, A Ekert, A. Kay and A. J. Landahl, \emph{Perfect transfer of arbitrary states in quantum spin networks}, Physical Review A, \textbf{71}:032312 (2005).

\bibitem{chi} W. Cheung and C. Godsil, \emph{Perfect state transfer in cubelike graphs}, Linear Algebra and Its Applications, \textbf{435}(10):2468-2474 (2011).

\bibitem{cou} G. Coutinho and C. Godsil, \emph{Perfect state transfer in products and covers of graphs}, Linear and Multilinear Algebra 64.2:235-246 (2015).

\bibitem{cou1} G. Coutinho and C. Godsil, \emph{Perfect state transfer is poly-time}, arXiv preprint arXiv:1606.02264 (2016).

\bibitem{cou2} G. Coutinho, C. Godsil, K. Guo and F. Vanhove, \emph{Perfect state transfer on distance-regular graphs and association schemes}, Linear Algebra and its Applications, \textbf{478}:108-130 (2015).

\bibitem{ek} A. K. Ekert, \emph{Quantum cryptography based on Bell's theorem}, Physical Review Letters, \textbf{67}(6):661 (1991).

\bibitem{esc} J. P. Escofier, \emph{Galois Theory}, 1nd ed. New York: Springer-Verlag (2001).

\bibitem{fan} X. Fan and C. Godsil, \emph{Pretty good state transfer on double stars}, Linear Algebra and Its Applications, \textbf{438}(5):2346-2358 (2013).

\bibitem{god1} C. Godsil, \emph{State transfer on graphs}, Discrete Mathematics (2011).

\bibitem{god2} C. Godsil, \emph{When can perfect state transfer occur?} Electronic Journal of Linear Algebra, 23:877-890 (2012).

\bibitem{god3} C. Godsil, \emph{Periodic graphs}, Electron. J. Combin., 18(1): Paper 23, 15 (2011).

\bibitem{god4} C. Godsil, S. Kirkland, S. Severini and J. Smith, \emph{Number-theoretic nature of communication in quantum spin systems}, Physical review letters 109, no. 5: 050502 (2012).

\bibitem{kirk} S. Kirkland, \emph{Sensitivity analysis of perfect state transfer in quantum spin networks}, Linear Algebra and its Applications 472: 1-30 (2015).

\bibitem{mala} G. Malajovich, {An Effective Version of Kronecker's Theorem on Simultaneous Diophantine Approximation.} Instituto de Matem\'atica da Universidade Federal do Rio de Janeiro, Brasil (2001).

\bibitem{pal} H. Pal and B. Bhattacharjya, \emph{Perfect state transfer on NEPS of the path on three vertices}, Discrete Mathematics \textbf{339}(2): 831-838, (2016).

\bibitem{pal1} H. Pal and B. Bhattacharjya, \emph{Perfect State Transfer on gcd-graphs}, accepted in Linear and Multilinear Algebra (2016).

\bibitem{pal2} H. Pal and B. Bhattacharjya, \emph{A class of gcd-graphs having Perfect State Transfer}, arXiv preprint arXiv:1601.07398 (2016).
  	
\bibitem{pal3} H. Pal and B. Bhattacharjya, \emph{Pretty Good State Transfer on Some NEPS}, arXiv preprint arXiv:1604.08858 (2016).

\bibitem{so} W. So, \emph{Integral circulant graphs}, Discrete Mathematics, \textbf{306}(1): 153-158 (2006).

\bibitem{wal1} W. Klotz and T. Sander, {Integral Cayley graphs over abelian groups.} Electron. J. Combin 17.1: R81 (2010).
\end{thebibliography}
\end{document}